\documentclass[11pt]{amsart}

\usepackage{amsmath,amssymb,amsthm,mathtools}
\usepackage{tikz}
\usetikzlibrary{arrows.meta}
\usepackage[onehalfspacing]{setspace}
\usepackage{hyperref} 
\usepackage{orcidlink}
\usepackage{enumitem}
\usepackage[noabbrev,nameinlink,capitalize]{cleveref}

\usepackage[portrait, margin=1in]{geometry}

\newcommand*\compl[1]{\overline{#1}}

\newtheorem{theorem}{Theorem}[section]
\AddToHook{env/theorem/begin}{\crefalias{section}{theorem}}

\newtheorem{corollary}[theorem]{Corollary}
\AddToHook{env/corollary/begin}{\crefalias{theorem}{corollary}}

\newtheorem{lemma}[theorem]{Lemma}
\AddToHook{env/lemma/begin}{\crefalias{theorem}{lemma}}

\newtheorem{proposition}[theorem]{Proposition}
\AddToHook{env/proposition/begin}{\crefalias{theorem}{proposition}}

\theoremstyle{definition}
\newtheorem{definition}[theorem]{Definition}
\AddToHook{env/definition/begin}{\crefalias{theorem}{definition}}

\newtheorem{example}[theorem]{Example}
\AddToHook{env/example/begin}{\crefalias{theorem}{example}}

\newtheorem{remark}[theorem]{Remark}
\AddToHook{env/remark/begin}{\crefalias{theorem}{remark}}

\begin{document}

\title{Graph sensitivity under join and decomposition}
\author{Cathy Kriloff}%$^{\ast}$}
\address{%$^{\ast}$Corresponding author \\
Professor Emerita, Department of Mathematics and Statistics\\
Idaho State University \\
Pocatello, ID 83209-8085}
\email{cathykriloff@isu.edu}
\author{Jacob Tolman}
\address{
Department of Mathematics and Computer Science \\
Wesleyan University \\
Middletown, CT 06469}
\email{jtolman@wesleyan.edu}

\begin{abstract}
The sensitivity, $\sigma(G)$, of a finite undirected simple graph $G$ is the smallest maximum degree of an induced subgraph on more than the maximum number of independent vertices.
Call an indexed family of graphs $G_n$, with maximum degree $\Delta(G_n) \to \infty$ as $n \to \infty$, sensitive if $\sigma(G_n) \to \infty$ and insensitive otherwise.
These definitions have their roots in Huang's resolution of the Sensitivity Conjecture for Boolean functions via determining sensitivity for the hypercube graphs and subsequent investigations of other Cayley graph families.
Here we describe sensitivity under the join operation and decomposition into stable blocks and construct sensitive and insensitive, primarily non-regular, graph families.
We determine the sensitivity explicitly for numerous singly- and doubly-indexed graph families, including certain generalized joins  - e.g., complete multipartite graphs and some generalized windmill graphs; general rooted products; and families of corona graphs. 
\end{abstract}

\date{}



\maketitle

\section{Introduction}
\label{intro}
The idea of graph sensitivity traces back to work 
related to Boolean functions and the hypercube graph (\cite{Gotsman-Linial-92,CFGS-88}), but has been formalized and studied for other graphs following Huang's clever and concise resolution in~\cite{Huang-19} of the sensitivity conjecture for Boolean functions~\cite{Nisan-Szegedy-94}.
The hypercube graph, $Q_n$, which has vertices the $n$-tuples in $\mathbb{Z}_2^n$ with two vertices adjacent if they differ in exactly one coordinate, is bipartite so its \textit{independence number} is $\alpha(Q_n)=\frac{1}{2} 2^n=2^{n-1}$.
Huang proved the smallest maximum degree of an induced subgraph of $Q_n$ on $\alpha(Q_n)+1$ vertices is at least $\sqrt{n}$ and the construction in~\cite{CFGS-88} implies it is equal to $\left\lceil \sqrt{n} \,\right\rceil$.
Huang further suggests studying, for other highly symmetric graphs, the smallest maximum degree of an induced subgraph of a graph $G$ on $\alpha(G)+1$ vertices, defined as the \textit{sensitivity} $\sigma(G)$ in~\cite{GarciaMarco-Knauer-22} by Garc\`{i}a-Marco and Knauer.

Since $Q_n$ is the Cayley graph on $\mathbb{Z}_2^n$ with connection set $S=\{e_1,\dots,e_n\}$ it is natural to investigate the sensitivity of other Cayley graphs.  For example, in~\cite{Alon-Zheng-20} Alon and Zheng proved $\sigma(\mathrm{Cay}(\mathbb{Z}_2^n,S)) \geq \sqrt{|S|}$ for any connection set $S$.  In~\cite{Potechin-Tsang-20}, Potechin and Tsang proved a corresponding result for all Cayley graphs of abelian groups using $|V_G|/2$ rather than $\alpha(G)$ as the cutoff for the number of vertices, so they prove that $\sqrt{|S|/2}$ provides a lower bound on sensitivity for bipartite graphs.  They conjectured that for every Cayley graph $\mathrm{Cay}(G,S)$, the maximum degree of an induced subgraph on more than half the vertices is at least $\sqrt{|S|/2}$, but Lehner and Verret provided an infinite family of counterexamples in~\cite{Lehner-Verret-20}.
Sensitivity and related properties have also been investigated by several authors for the natural generalization of hypercubes to Hamming graphs $H_{n,q}=\mathrm{Cay}(\mathbb{Z}_q^n,\{\pm e_1,\dots,\pm e_n\})$ for $q \geq 3$ (see Theorem 8.1 in~\cite{GarciaMarco-Knauer-22} as well as~\cite{Dong-21,Tandya-22,Potechin-Tsang-24,AFGMK-25}).

The definition of sensitivity of an indexed graph family and three additional infinite families of counterexamples to the Potechin and Tsang conjecture appear in~\cite{GarciaMarco-Knauer-22} by Garc\`{i}a-Marco and Knauer.  All of these families have unbounded degree but sensitivity equal to one, which they refer to as \textit{insensitive} families. 
The Hamming graph families $\{H_{n,q}\}_{n=1}^{\infty}$ for $q \geq 3$ and $\{H_{n,q}\}_{q=3}^{\infty}$ are also insensitive since $\Delta(H_{n,q})=(q-1)n$ and for $q \geq 3$ Tandya constructs an induced subgraph of maximum degree one on $\alpha(H_{n,q})+1$ vertices in~\cite{Tandya-22}.  

In this paper we formalize definitions from~\cite{GarciaMarco-Knauer-22} and~\cite{Huang-19}, describe graph sensitivity under the join operation and a partition into stable blocks, and apply these results to give general constructions of sensitive and insensitive graph families.  We investigate the sensitivity of several explicit, mostly non-regular, families, some of which serve as models for citation, transportation, or neural networks, and find sensitive and insensitive subfamilies occurring within singly- and doubly-indexed families of graphs in a variety of ways. We are aware of two other references that consider sensitivity of nonregular families using a graph operation. Both use a spectral approach as in~\cite{Huang-19} to analyze Cartesian products: of connected signed graphs one of which is bipartite in \cite{HLL-21}, and of paths in~\cite{ZH-24}.

More details of the contents of this paper are as follows.
In~\cref{defs} after preliminary definitions we present a few examples of sensitive and insensitive graph families, and the definition of \textit{$k$-sensitivity}, $\sigma_k(G)$, of a nonempty graph $G$ for $1 \leq k \leq |V_G|-\alpha(G)$ as the minimum of the maximum degrees of induced subgraphs on $\alpha(G)+k$ vertices (see the second concluding remark in~\cite{Huang-19} and Question (5.1) in~\cite{ZH-24}).

Results about joins are in~\cref{joins}.  We describe the sensitivity of the join of any two graphs in~\cref{sensitivity-join-edges}, and apply this to the join of two or more copies of a graph (\cref{doublegraph}) 
and to the join with a complete graph (\cref{Kn-join-G}). In~\cref{n-cone gensensitivity} we specialize to describe the sensitivity of the join of an empty graph, $\compl{K}_n$, with any graph $G$ in terms of $n$, the $k$-sensitivities of $G$, and/or $\alpha(G)$ - depending on how $n$ compares to $|V_G|$ and $\alpha(G)$ - including further refining $\sigma(\compl{K}_n \vee G)$ when $n\leq \alpha(G)$.  The results about $\sigma(\compl{K}_n \vee G)$ are summarized in~\cref{number line}.  In addition, when $G$ is a regular graph of degree $d$ on $m$ vertices and $n\geq 2m-d-1$ we show $\sigma(\compl{K}_n \vee G)=n-m+d+1$ (\cref{n-cone-lower-bound}).
\cref{joins} also includes some general constructions and numerous examples.  For instance, we show that joining a family of empty graphs to a graph (\cref{sens-n-cones}) and joining an empty graph to a sensitive family (\cref{n-cone-sens-families}) both result in a sensitive family.
We compute sensitivity and/or determine (in)sensitivity for several indexed families, including certain generalized joins - e.g., complete bipartite and multipartite graphs (\cref{complete-bipartite-sens,complete-multipartite-sens,complete-reg-multipartite}) and some generalized windmill graphs (\cref{windmill,genwindmill}).

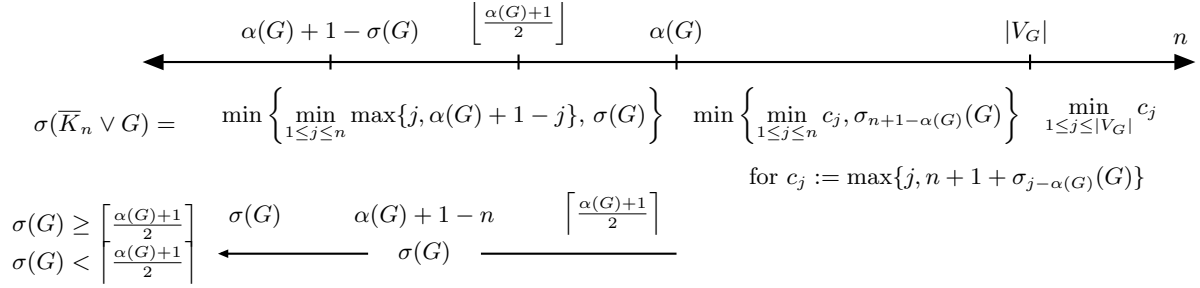
\begin{figure}
%\hspace*{\dimexpr-\textwidth/5}
\centering
  \begin{tikzpicture}

    % draw the number line 
    \draw[thick, <->, >={Stealth[inset=0pt, length=8pt]}] (-7,0) -- (7,0);

    % draw tick marks
    \foreach \x in {-4.5,-2,0.1,4.8} {
        \draw[thick] (\x,0.1) -- (\x,-0.1);
      }

    % labels above tick marks
    \node[above] at (-4.5,0.1) {\footnotesize{$\alpha(G)+1-\sigma(G)$}};
    \node[above] at (-2,0.1) {\footnotesize{$\left\lfloor\frac{\alpha(G)+1}{2}\right\rfloor$}};
    \node[above] at (0.1,0.1) {\footnotesize{$\alpha(G)$}};

    \node[above] at (4.75,0.1) {\footnotesize{$|V_G|$}};

    \node[above] at (6.8,0.1) {\footnotesize{$n$}};

    % label below tick marks
    \node[below, text = %blue!80!white
    ] at (-7.5,-0.5) {\footnotesize{$\sigma(\compl{K}_n \vee G)=$}};
    \node[below, text = %blue!80!white
    ] at (-3,-0.2) {\footnotesize{$\min \left\{ \min\limits_{1 \leq j \leq n} \max \{j, \alpha(G)+1-j\},  \,\sigma(G)\right\}$}};
    \node[below, text = %blue!80!white
    ] at (2.5, -0.2) {\footnotesize{$\min \left\{ \min\limits_{1 \leq j \leq n} c_j, \sigma_{n+1 - \alpha(G)}(G) \right\}$}};
    \node[below, text = %blue!80!white
    ] at (5.75, -0.35) {\footnotesize{$\min\limits_{1 \leq j \leq |V_G|} c_j$}};

    \node[below] at (-7.5,-1.75) {\footnotesize{$\sigma(G)\geq\left\lceil\frac{\alpha(G)+1}{2}\right\rceil$}};
    \node[below, text = %blue!80!white
    ] at (-5.5,-1.75) {\footnotesize{$\sigma(G)$}};
    \node[below, text = %blue!80!white
    ] at (-3.25,-1.75) {\footnotesize{$\alpha(G)+1-n$}};
    \node[below, text = %blue!80!white
    ] at (-0.75,-1.6) {\footnotesize{$\left\lceil\frac{\alpha(G)+1}{2}\right\rceil$}};
    \node[below%, text = blue!80!white
    ] at (3.7,-1.25) {\footnotesize{for $c_j:=\max\{j,n+1-j+\sigma_{j-\alpha(G)}(G)\}$}};

    \node[below] at (-7.5,-2.25) {\footnotesize{$\sigma(G)<\left\lceil\frac{\alpha(G)+1}{2}\right\rceil$}};
    \node[below, text = %blue!80!white
    ] at (-3.25,-2.25) {\footnotesize{$\sigma(G)$}};

    % draw arrows
    \draw[thick, <-, %green!50!gray, 
    >={Stealth[inset=0pt, length=4pt]}] (-6,-2.55) -- (-4,-2.55);
    \draw[thick, -%, green!50!gray
    ] (-2.5,-2.55) -- (0.1,-2.55);
  \end{tikzpicture}
\caption{Summary of values for sensitivity of $\compl{K}_n \vee G$ in~\cref{n-cone gensensitivity}.  \\
(If $\sigma(G)<\frac{\alpha(G)+1}{2}$ then $\alpha(G)+1-\sigma(G)$ and $\left\lfloor\frac{\alpha(G)+1}{2}\right\rfloor$ are interchanged.)}
\label{number line}
\end{figure}

In~\cref{stable blocks} we instead consider examples like rooted products and coronas that allow a decomposition into stable blocks -- sets of vertices for which the independence number is determined on the interior of the subgraph they induce (\cref{sb-def}). In~\cref{stable-block-partition} we reduce the computation of sensitivity to the stable blocks when they partition a graph.

\medskip
\noindent
\textbf{\cref{stable-block-partition}}
\textit{If the vertex set of a graph $G$ can be partitioned into stable blocks $V_1,\dots,V_n$ then $\sigma(G)=\min\limits_{1 \leq i \leq n} \sigma(G[V_i])$.}
\medskip

\noindent While a partition into stable blocks is likely quite rare, we provide a construction that results in such a partition and can be used to build a variety of examples, e.g., by using \textit{rooted products} (\cref{rp-def}) in which each root avoids some maximum independent set.  
We show the sensitivity of a general rooted product is at most the minimum sensitivity of its components (\cref{rooted-product}).
For a rooted product of a graph $G$ with copies of $K_1 \vee H$ for some graph $H$ - i.e., for the \textit{corona} $G \odot H$ - 
we show in~\cref{corona-sens} that when $H$ is nonempty $\sigma(G \odot H) = \sigma(H)$ and when $H$ is empty $\sigma(G \odot H)=|V_H|$.
This reduction allows for computing sensitivity of families built using the corona operation from families with known sensitivity and constructing a family with a prescribed sequence of sensitivities.  

\cref{Concluding remarks} provides some context on our results and possible further directions.

\section{Preliminaries}
\label{defs}

We will be working with \textit{finite undirected simple} graphs $G$ and call the number of vertices in $G$ its \textit{order}. For $n \geq 1$ we will denote the \textit{complete graph} on $n$ vertices as $K_n$ and the \textit{empty graph} containing no edges as $\compl{K}_n$. For convenience we write $K_1$ for the graph with one vertex and no edges. We denote the \emph{maximum degree} of $G$ by $\Delta(G)$, or $\Delta$ when $G$ is understood, and by $d$ when $G$ is \emph{regular} with degree $d$ at each vertex.

\begin{definition}
  If $G = (V,E)$ is a graph and $S \subseteq V$ is a subset of vertices, the \emph{induced subgraph} with vertex set $S$ (or \emph{subgraph induced by $S$}), denoted $G[S]$, is the graph with vertex set $S$, and an edge between $u$ and $v$ if and only if $u$ and $v$ are adjacent in $G$.  At times, for simplicity of notation $S$ might be used if the context makes it clear that $G[S]$ is meant. 
\end{definition}

The well-studied \emph{independence number} of a graph $G$, denoted $\alpha(G)$, is the maximum order of an empty induced subgraph of $G$.  The sensitivity of $G$ measures the minimum amount by which the maximum degree must increase when considering induced subgraphs on more than $\alpha(G)$ vertices.

\begin{definition}
  \label{sens}
  The \emph{sensitivity} of a nonempty graph $G = (V,E)$ is the number
  \[
    \sigma(G) \coloneqq \min \{\Delta(G[S]) : S \subseteq V \text{ and } |S| > \alpha(G)\}\,,
  \]
  and we define $\sigma(\compl{K}_n)=n$.\footnote{We thank Martin Rubey for proposing $\sigma(\compl{K}_n)=n$ as more natural than $\infty$ since then adding a dominant vertex to an empty graph preserves sensitivity (see~\cref{star}), and for adding graph sensitivity to the database findstat.org.}
  An indexed family of graphs $G_n$ such that $\Delta(G_n) \to \infty$ as $n \to \infty$ is \emph{sensitive} if also $\sigma(G_n)\to \infty$ and is \emph{insensitive} if not.
\end{definition}

The definition of the sensitivity of a graph and what it means for an indexed graph family to be sensitive or insensitive are motivated by Huang's result~\cite{Huang-19} on hypercubes and the discussion in~\cite{GarciaMarco-Knauer-22}.
To compute $\sigma(G)$ for a nonempty graph it suffices to consider vertex sets with $|S|=\alpha(G)+1$ since the maximum degree can only increase when using a larger vertex set, but all possible such sets must be considered, not just those formed by adding one vertex to a maximum order empty induced subgraph.  
Defining $\sigma(\compl{K}_n)$ to be $n$ allows us to state~\cref{sensitivity-join-edges} uniformly for all graphs.
However, we do not consider the family of empty graphs to be sensitive or insensitive and in general the condition that $\Delta(G_n) \to \infty$ as $n \to \infty$ distinguishes insensitive graph families from those where $\sigma(G_n)$ is bounded because $\Delta(G_n)$ is bounded.

\begin{example}
\label{Kn-Pn-Cn}
  The family of \textit{complete graphs} $K_n$ on $n \geq 2$ vertices is insensitive since
  \[\Delta(K_n)=n-1, \hbox{ and } \alpha(K_n)=1 \hbox{ implies } \sigma(K_n)=1.\]
However, the family of \textit{path graphs} $P_n$ for $n \geq 4$ has
  \[\alpha(P_n)=\left\lceil \frac{n}{2} \right\rceil
    \hbox{ and }
    \sigma(P_n)=1\]
and the \textit{cycle graphs}, $C_n$, for $n \geq 5$ have
  \[
    \alpha(C_n)=\left\lfloor \frac{n}{2} \right\rfloor
    \hbox{ and }
    \sigma(C_n)=1\,,\]
but neither of these families is considered insensitive since $\Delta(P_n)=\Delta(C_n)=2$. 
\end{example}

The initial motivating example of a sensitive family is the hypercube family $Q_n$, for which $\sigma(Q_n)\geq \left\lceil \sqrt{n} \,\right\rceil$~\cite{Huang-19,CFGS-88}.
To introduce two simple examples of sensitive non-regular families, recall that the \textit{join} $G \vee H$ of two graphs $G$ and $H$ is the disjoint union of the two graphs together with all possible edges between them. A natural and useful first example is star graphs.  

\begin{example}
  \label{star}
  The \textit{star graphs} $K_{1,n}=K_1\vee \compl{K}_n$ are formed by taking the join with a single vertex, or the \textit{cone}, of a complete graph and are easily seen to form a sensitive family with
  \[\Delta(K_{1,n})=\alpha(K_{1,n})=\sigma(K_{1,n})=n.\]
\end{example}

The following lemma will be helpful when analyzing the sensitivity of joins.

\begin{lemma}
  \label{independence-num-join}
  Given graphs $G_1$ and $G_2$, for $G_1 \vee G_2$ we have
  \begin{align*}
    \Delta(G_1 \vee G_2) & = \max \{\Delta(G_1)+|V_{G_2}|,\Delta(G_2)+|V_{G_1}|\}, \hbox{ and} \\
    \alpha(G_1 \vee G_2) & = \max \{\alpha(G_1),\alpha(G_2)\}.
  \end{align*}
\end{lemma}
\begin{proof}
  These follow because in $G_1 \vee G_2$ each vertex in $G_1$ is adjacent to every vertex in $G_2$, and hence an independent subset of $G_1 \vee G_2$ must lie entirely within $G_1$ or entirely within $G_2$. 
\end{proof}

Taking the cone of the star graphs results in another sensitive family with sensitivity $n$.

\begin{example}
  \label{1-cone-example-sens}
  For $n \geq 1$ the complete tripartite graphs $K_{1,1,n}=K_1 \vee (K_1 \vee \compl{K}_n)$\,, also called \textit{thagomizer graphs}, have
  \[\Delta(K_{1,1,n})=n+1, \alpha(K_{1,1,n})=n\,, \hbox{ and } \sigma(K_{1,1,n})=\sigma(K_1 \vee \compl{K}_n)=n\] by~\cref{star} and~\cref{independence-num-join}.  Thus $\{K_{1,1,n}\}_{n=1}^\infty$ is sensitive.
\end{example}

Analyzing the join of two arbitrary graphs will require the following definition of $k$-sensitivity.

\begin{definition}
  \label{k-sens}
  For any nonempty graph $G$ and any integer $1 \leq k \leq |V_G| - \alpha(G)$\,, define the \textit{$k$-sensitivity of $G$} as
  \[\sigma_k(G) \coloneqq \min \{\Delta H \colon H \text{ is an $(\alpha(G) + k)$-vertex induced subgraph of $G$} \}.
  \]
  For an empty graph define $\sigma_k(\compl{K}_n)=n$ for all $k \geq 1$ and for any graph $G$ define $\sigma_k(G)=0$ if $k \leq 0$.
\end{definition}

Note that for all graphs $\sigma_1(G) = \sigma(G)$ and for all nonempty graphs we have $\sigma_k(G)=\Delta_{\beta}(G)$ for $\beta=\frac{\alpha(G)+k}{|V_G|}$ when written in the notation of~\cite{GarciaMarco-Knauer-22}.

\begin{example}
  \label{gen-sens-Kn}
  Since an induced subgraph of a complete graph is complete, $\sigma_k(K_n)=k$ for all $n \geq 2$ and $1 \leq k \leq n-1$.
\end{example}

\begin{example}
  The thagomizer graph $K_{1,1,n}$ has $n+2$ vertices and $\alpha(K_{1,1,n})=n$\,, with the only two values of $k$-sensitivity being $\sigma(K_{1,1,n})=n$ as in~\cref{1-cone-example-sens} and $\sigma_2(K_{1,1,n})=\Delta(K_{1,1,n})=n+1$.
\end{example}

\begin{example}
For Hamming graphs with $q \geq 3$ the sensitivity $\sigma(H_{n,q})=1$ (see~\cite{Dong-21,Tandya-22,Potechin-Tsang-24}).
In~\cite{Potechin-Tsang-24} Potechin and Tsang find additional values of $k>1$ such that $\sigma_k(H_{n,3})=1$.
\end{example}

We also note that the $k$-sensitivity of the disjoint union of copies of the same graph can be reduced to finding the same or smaller sensitivity of that graph.

\begin{proposition}
  \label{k-sens-nH}
  For $n$ disjoint copies of any nonempty graph $G$ and for $1 \leq k \leq n|V_G|-n\alpha(G)$ the $k$-sensitivity is $\sigma_k(nG)=\sigma_{\ell}(G)$ where $\ell=\left\lceil \frac{k}{n} \right\rceil$.
\end{proposition}

\begin{proof}
  We first note that $\alpha(nG)=n\alpha(G)$.  Since the $k$-sensitivity is nondecreasing in $k$, the smallest maximum degree of an induced subgraph in $nG$ on $n\alpha(G)+k$ vertices will be realized by putting $\alpha(G)$ independent vertices in each copy of $G$, distributing the remaining $k$ vertices as evenly as possible among the copies, and arranging to realize the optimum induced subgraph within at least one copy of $G$ that contains the most additional vertices.
  Thus $\sigma_k(nG)=\sigma_{\ell}(G)$ where $\ell=\left\lceil \frac{k}{n} \right\rceil$. 
\end{proof}

The final lemma will be used to find the sensitivity of joins of graphs in the following sections.

\begin{lemma}
  \label{simple-minofmax}
  Let $r,s$ be positive integers.  Then
  \[\min\limits_{1 \leq j \leq s} \max\{j,r-j\}=
    \begin{cases}
      r-s                                & \text{if } s \leq \left\lfloor \frac{r}{2} \right\rfloor, \\
      \left\lceil\frac{r}{2}\right\rceil & \text{if } s > \left\lfloor \frac{r}{2} \right\rfloor.
    \end{cases}\]
\end{lemma}
\begin{proof}
  If $s\leq \left\lfloor \frac{r}{2} \right\rfloor$ then for $1 \leq j \leq s$ we have $\max\{j,r-j\}=r-j$ and the minimum of these occurs when $j=s$.
  If $s > \left\lfloor \frac{r}{2} \right\rfloor$ then for $1 \leq j \leq \left\lfloor \frac{r}{2} \right\rfloor$ we again have $\max\{j,r-j\}=r-j$ with the minimum of these being $r-\left\lfloor \frac{r}{2} \right\rfloor = \left\lceil\frac{r}{2}\right\rceil$.  But for $\left\lfloor \frac{r}{2} \right\rfloor < j \leq s$ we have $\max\{j,r-j\}=j$ and the minimum of these occurs when $j=\left\lceil \frac{r}{2} \right\rceil$. 
\end{proof}

\section{Sensitivity and join}
\label{joins}
In this section we analyze sensitivity under the join operation.  The results yield methods for constructing several examples of sensitive and insensitive graph families and subfamilies, including complete multipartite graphs and certain other forms of generalized joins.

\subsection{Join of two graphs}
\label{sens-join}

We find the sensitivity of the join of two arbitrary graphs in terms of the sensitivity of the one with larger independence number and some $k$-sensitivities of the one with smaller independence number. As a result, we obtain explicit formulas or reduction  formulas for the sensitivity of complete bipartite graphs, the join of two or more copies of the same graph, and a join with the complete graph, revealing both sensitive and insensitive families.

\begin{theorem}\label{sensitivity-join-edges}
  Let $G_1$ and $G_2$ be graphs with independence numbers $\alpha_1 \geq \alpha_2$ and vertex sets $V_1$ and $V_2$. Define
  \begin{align*}
    b_j & \coloneqq \max \{j, \alpha_1+1-j\} \hbox{ and }                      \\
    c_j & \coloneqq \max \{j, \alpha_1 + 1 - j + \sigma_{j - \alpha_2}(G_2)\}.
  \end{align*}

  \begin{enumerate}
    \item \label{join-1} If $|V_2|<\alpha_1+1$ then
          \[\sigma(G_1 \vee G_2)
            =\min\left\{\sigma(G_1)\,,\displaystyle\min_{1 \leq j \leq \alpha_2} b_j\,,\displaystyle\min_{\alpha_2+1 \leq j \leq |V_2|} c_j\right\}.\]
    \item \label{join-2} If $|V_2| \geq \alpha_1+1$ then
          \[\sigma(G_1 \vee G_2)=\min\left\{\sigma(G_1)\,,\displaystyle\min_{1 \leq j \leq \alpha_2} b_j\,,\displaystyle\min_{\alpha_2 +1 \leq j \leq \alpha_1} c_j\,,\,\sigma_{\alpha_1+1-\alpha_2}(G_2)\right\}.\]
  \end{enumerate}
\end{theorem}

\begin{proof}
  Let $G_1$ and $G_2$ be graphs with vertex sets $V_1$ and $V_2$ and independence numbers $\alpha_1 \geq \alpha_2$. Hence $\alpha(G_1 \vee G_2) = \alpha_1$. We will determine the quantities
  \[
    a_j \coloneqq
    \min \{\Delta(H) \colon  \text{$H$ an induced subgraph with $|V_H| = \alpha_1 + 1$ and $|V_{H} \cap V_2| = j$}\}
  \]
  with $0 \leq j \leq \min\{|V_2|,\alpha_1 + 1\}$ whenever they are defined.  Then $\sigma(G_1 \vee G_2)$ will be the minimum of these $a_j$.

  First suppose that $j = 0$. If $G_1$ is nonempty then
  \[
    a_0 = \min \{\Delta(H) \colon |V_H| = \alpha_1+1, \ |V_{H} \cap V_2| = 0\}
  \]
  is by definition $\sigma(G_1)$.
  Observe that if $G_1$ is empty, then $a_0$ is undefined, but in this case by~\cref{k-sens} $a_1=|V_1|=\sigma(G_1)$, so we use $a_0=\sigma(G_1)$ in this case as well.

  Next, consider $1 \leq j \leq \alpha_2$. Since all $j$ vertices of $H$ in $G_2$ are joined to the remaining $\alpha_1+1-j$ vertices of $H$ in $G_1$, 
  to minimize $\Delta(H)$ we use $j$ independent vertices in $G_2$ and $\alpha_1+1-j$ independent vertices in $G_1$.  Thus for $1 \leq j \leq \alpha_2$ and such a subgraph $H$,
  \[
    a_j = \Delta(H) = \max \{j, \alpha_1+1-j\}.
  \]

  Finally, suppose that $\alpha_2+1 \leq j \leq \min \{|V_2|, \alpha_1+1\}$. Note that if $G_2$ is empty then $|V_2|=\alpha_2\leq \alpha_1$ and there is no such $j$. So assume that $G_2$ is not empty. In this case,  an induced subgraph $H$ of $G_1 \vee G_2$ on $\alpha_1+1$ vertices with $|V_H \cap V_2|=j$ must have at least one edge in $G_2$. Thus, in order to find $a_j$ we must consider the sensitivity $\sigma_{j - \alpha_2}(G_2)$.

  If $|V_2| < \alpha_1+1$ then there will be at least one vertex in $G_1$ and to minimize $\Delta(H)$ the vertices of $H$ in $G_2$ must be chosen so that $\Delta(H\cap G_2)=\sigma_{j-\alpha_2}(G_2)$, i.e., to have the subgraph induced by $H$ within $G_2$ be of smallest maximum degree.  Then for $\alpha_2+1 \leq j \leq |V_2|$ and such a subgraph $H$, 
    \[
    a_j = \Delta(H) = \max \{j, \alpha_1+1-j+\sigma_{j-\alpha_2}(G_2)\}.
  \]
  
  If $|V_2| \geq \alpha_1+1$, then for $\alpha_2 + 1 \leq j \leq \alpha_1$ the same argument shows that
  \[
    a_j = \max \{j, \alpha_1+1-j+\sigma_{j-\alpha_2}(G_2)\}.
  \]
But when $j=\alpha_1+1\leq |V_2|$, then every induced subgraph used in computing $a_j=a_{\alpha_1+1}$ has all vertices in $G_2$ and in this case, $a_{\alpha_1+1}=\sigma_{\alpha_1+1-\alpha_2}(G_2)$.
  
  Taking the minimum of all the quantities $a_j$ and introducing the notation $b_j$ and $c_j$ results in the given values for $\sigma(G_1 \vee G_2)$.
\end{proof}

It follows from the proof that in~\cref{sensitivity-join-edges}~(\ref{join-1}) and~(\ref{join-2}), when $G_1$ is empty $b_1=|G_1|=\sigma(G_1)$ by~\cref{k-sens} and $\sigma(G_1)$ may either be included or omitted, and when $G_2$ is empty the $c_j$ may be omitted.
This combined with~\cref{simple-minofmax} means that when $G_1$ and $G_2$ are both empty the expressions in~\cref{sensitivity-join-edges} yield an explicit formula for the sensitivity of the \textit{complete bipartite graphs} $K_{m,n}=\compl{K}_m \vee \compl{K}_n$ with $m \leq n$.  This can also be found directly and amounts to splitting $n+1$ vertices as evenly as possible between the two empty graphs, except that when $m \leq \left\lfloor \frac{n+1}{2} \right\rfloor$ it is optimal to include all $m$ vertices in $\compl{K}_m$.

\begin{corollary}\label{complete-bipartite-sens}
  The family of complete bipartite graphs $\compl{K}_m \vee \compl{K}_n$ with $m \leq n$ has sensitivity
  \[
    \sigma(\compl{K}_m \vee \compl{K}_n) = \begin{cases}
      n+1-m                                  & \text{ if $m \leq \left\lfloor \frac{n+1}{2}\right\rfloor$\,,} \\
      \left\lceil \frac{n+1}{2} \right\rceil & \text{ if $m > \left\lfloor \frac{n+1}{2}\right\rfloor$.}
    \end{cases}
  \]
Thus the family $\{\compl{K}_m \vee \compl{K}_n\}_{n=1}^\infty$ is sensitive and provided $\left\lfloor \frac{n+1}{2}\right\rfloor <m \leq n$ stays true the family $\{\compl{K}_m \vee \compl{K}_n\}_{m,n=1}^\infty$ is also sensitive.
In particular, the \textit{complete regular bipartite graph} $\compl{K}_n \vee \compl{K}_n$ has sensitivity
  \[
    \sigma(\compl{K}_n \vee \compl{K}_n) = \left\lceil \frac{n+1}{2} \right\rceil,
  \]
  and the family of complete regular bipartite graphs $\{\compl{K}_n \vee \compl{K}_n\}_{n=1}^\infty$ is sensitive.
\end{corollary}

The sensitivity of the \textit{double graph}, $G \vee G$, of $G$, or more generally the join $G \vee \cdots \vee G$ of at least two copies of $G$, is also easily analyzed using~\cref{sensitivity-join-edges}. 

\begin{corollary}
  \label{doublegraph}
  If $G$ is any graph then for the join of at least two copies of $G$ we have 
  \[\sigma(G \vee \cdots \vee G)=\min\left\{\sigma(G),\left\lceil \frac{\alpha(G)+1}{2} \right\rceil\right\}.\]
  Thus for an indexed family $\{G_n\}_{n=1}^{\infty}$ the family $\left\{G_n \vee \cdots \vee G_n\right\}_{n=1}^{\infty}$ is sensitive if $\{G_n\}_{n=1}^{\infty}$ is sensitive, and is insensitive if $\{G_n\}_{n=1}^{\infty}$ is insensitive.  
\end{corollary}

\begin{proof}
For the double graph $G \vee G$ with $G$ nonempty this follows from~\cref{sensitivity-join-edges}~(\ref{join-2}) and \cref{simple-minofmax} because $\alpha_2=\alpha_1=\alpha(G)$ so the minimum in~\cref{sensitivity-join-edges}~(\ref{join-2}) only involves the two quantities $\sigma(G)$ and $\left\lceil \frac{\alpha(G)+1}{2} \right\rceil$.
When $G=\compl{K}_m$, then this follows  by~\cref{complete-bipartite-sens}.
In fact, since $\alpha(G \vee G)=\alpha(G)$, repeatedly joining $G$ with itself more than once will yield the same sensitivity as for the double graph of $G$. 
\end{proof}

\cref{sensitivity-join-edges} also yields that joining a complete graph to a nonempty graph preserves sensitivity and that the sensitivity of the \textit{complete split} graph $\compl{K}_m\vee K_n$ is $m$.

\begin{corollary}
  \label{Kn-join-G}
  If $G$ is a graph and $n \geq 1$, then $\sigma(G \vee K_n)=\sigma(G)$
  and the family $\{G \vee K_n\}_{n=1}^{\infty}$ is insensitive while $\{\compl{K}_m \vee K_n\}_{m=1}^{\infty}$ is sensitive.
\end{corollary}

\begin{proof}
  If $G$ is nonempty, then using that $\sigma_k(K_n)=k$ for all $n \geq 2$ and $1 \leq k \leq n-1$ by \cref{gen-sens-Kn} and $K_n$ as $G_2$ in~\cref{sensitivity-join-edges} yields that 
  \[
  \sigma(G\vee K_n)=\min\{\sigma(G),\alpha(G)\}=\sigma(G). 
  \]
  If $n=1$ then~\cref{sensitivity-join-edges}~(\ref{join-1}) again reduces to $\sigma(G\vee K_1)=\sigma(G)$.
  
  If $G=\compl{K}_m$, then both~\cref{sensitivity-join-edges}~(\ref{join-1}) and~(\ref{join-2}) yield 
  \[
    \sigma(G \vee K_n) = m = \sigma(G).
  \]
  
  By~\cref{independence-num-join}, we know that as $n \to \infty$,
  \[
  \Delta(G\vee K_n)=\max\{n-1+|V_G|,n+\Delta(G)\} \to \infty. 
  \] 
  We conclude that $\{\compl{K}_m \vee K_n\}_{m=1}^{\infty}$ is sensitive and $\{G \vee K_n\}_{n=1}^{\infty}$ is insensitive.
\end{proof}

\cref{Kn-join-G} can be used to show that taking the cone of the path and cycle families or of any disjoint union of complete graphs produces insensitive families.
\begin{example}
  \label{1-cone-examples-insens}
  It follows from~\cref{Kn-join-G} and~\cref{Kn-Pn-Cn}, but is also simple to verify directly, that for $m \geq 4$ the \textit{fan graphs} $K_1 \vee P_m$ and for $m\geq 5$ the \textit{wheel graphs} $K_1 \vee C_m$ have
  \[\sigma(K_1 \vee P_m)=\sigma(K_1 \vee C_m)=1.\]
  Since $\Delta(K_1 \vee P_m)=\Delta(K_1 \vee C_m)=m$,
  both $\{K_1 \vee P_m\}_{m=1}^{\infty}$ and $\{K_1 \vee C_m\}_{m=1}^{\infty}$ are insensitive.
\end{example}

\begin{example}
  \label{windmill}
  Consider the \textit{windmill graphs} $W_{m,n}=K_1 \vee mK_n$, where $mK_n$ is $m$ disjoint copies of $K_n$.
By~\cref{Kn-join-G} we know that $\sigma(K_1 \vee G)=\sigma(G)$ and taking the $1$-cone of a graph preserves sensitivity.  
Hence, provided $n>1$, all of the subfamilies of the windmill graphs are insensitive with $\Delta(K_1 \vee mK_n)=mn$ and $\sigma(K_1 \vee mK_n)=\sigma(mK_n)=1$ while $\sigma(K_1 \vee mK_1)=\sigma(K_1 \vee \overline{K}_m)=m$ as in~\cref{star}.
The specific windmill graphs $K_1 \vee mK_2$ are often referred to as \textit{Dutch windmill} or \textit{friendship} graphs.
\end{example}

\subsection{Joining an empty graph to a nonempty graph}
\label{sens-join-Kn-G}

We refer to the join of $G$ and the empty graph on $n$ vertices as the \textit{$n$-cone of $G$}. In this section, we use~\cref{k-sens} to restate \cref{sensitivity-join-edges} in the case where only one of the graphs is empty. This yields more explicit expressions for the sensitivity of the $n$-cone of $G$, allowing exploration of the sensitivity for an array of graph families, including generalized joins like complete multipartite graphs and certain generalized windmill graphs.

\begin{theorem}
  \label{n-cone gensensitivity}
  Let $G$ be any graph and consider the $n$-cone $\compl{K}_n \vee G$ for $n \geq 1$.
  Let \[c_j \coloneqq \max \{j, n + 1 - j + \sigma_{j - \alpha(G)}(G)\}.\]

  \begin{enumerate}
    \item
          \label{item1}
          If $n \geq |V_G|$ then
          \[\sigma(\compl{K}_n \vee G)=\min\limits_{1 \leq j \leq |V_G|} c_j.\]

    \item
          \label{item2}
          If $\alpha(G) < n <|V_G|$ then
          \[\sigma(\compl{K}_n \vee G)=\min \left\{ \min\limits_{1 \leq j \leq n} c_j, \sigma_{n+1 - \alpha(G)}(G) \right\}.\]

    \item
          \label{item3}
          If $n \leq \alpha(G)$ then 
          \begin{align*}
          \sigma(\compl{K}_n \vee G) &= \min \left\{ \min\limits_{1 \leq j \leq n} \max \{j, \alpha(G)+1-j\},  \,\sigma(G)\right\} \\
          &=
    \begin{cases}
      \min \left\{\sigma(G), \left\lceil \frac{\alpha(G)+1}{2} \right\rceil\right\} \  & \text{if $n > \left\lfloor \frac{\alpha(G)+1}{2} \right\rfloor$\,,} \\       \min \left\{\sigma(G), \alpha(G)+1-n  \right\} \ &\text{if $n \leq \left\lfloor \frac{\alpha(G)+1}{2} \right\rfloor$.}
    \end{cases}
    \end{align*}
  \end{enumerate}

\end{theorem}

\begin{proof}
  This follows directly from \cref{sensitivity-join-edges} by comparing $n=\alpha(\compl{K}_n)$ with $\alpha(G)$ to determine whether $G$ or $\compl{K}_n$ serves as $G_1$ in that statement.

  In cases~(\ref{item1}) and~(\ref{item2}), $\alpha(\compl{K}_n) > \alpha(G)$, so use $\compl{K}_n$  as $G_1$ and $G$ as $G_2$. The statements in these cases follow from the corresponding parts of \cref{sensitivity-join-edges}, noting that when $j \leq \alpha_2$ since $\sigma_{j-\alpha_2}(G)=0$ then $b_j=c_j$, and that $c_1=\max\{1,n+\sigma_{1-\alpha(G)}(G)\}=\max\{1,n\}=n$, so $\sigma(\compl{K}_n)=n$ may be omitted when taking the minimum.

  In case~(\ref{item3}) since $\alpha(\compl{K}_n)=|\compl{K}_n| \leq \alpha(G)$, instead $\compl{K}_n$ serves as $G_2$, $G$ serves as $G_1$, and only the first part of~\cref{sensitivity-join-edges} applies, which yields the first expression for $\sigma(\compl{K}_n \vee G)$.
  We can also apply~\cref{simple-minofmax} to more precisely describe the sensitivity as in the second expression.  In particular, if $n=1$ then we recover $\sigma(K_1 \vee G)=\sigma(G)$.
\end{proof}

As a first consequence of~\cref{n-cone gensensitivity}, using only that all $k$-sensitivities are positive we can produce sensitive families by joining any fixed graph to $\compl{K}_n$.

\begin{corollary}
  \label{sens-n-cones}
  If $G$ is any graph then the family $\{\compl{K}_n \vee G\}_{n=1}^\infty$ of $n$-cones is sensitive.
\end{corollary}

\begin{proof}
  Using $0<\sigma_k(G)$ for all $1 \leq k \leq |V_G| - \alpha(G)$ in~\cref{n-cone gensensitivity}~(\ref{item1}), we know that whenever $n \geq |V_G|$ then
  \begin{align*}
    \sigma(\compl{K}_n \vee G) & =\min \left\{\max \{j,\, n+1-j+\sigma_{j-\alpha(G)}(G)\} : 1 \leq j \leq |V_G|\right\} \\
                        & \geq \min \left\{\max \{j,\, n+1-j\} : 1 \leq j \leq |V_G|\right\}.
  \end{align*}
  By~\cref{simple-minofmax}, once $n$ is sufficiently larger than $|V_G|$, we can conclude that
  $\sigma(\compl{K}_n \vee G)\geq n+1-|V_G|$ and hence $\{\compl{K}_n \vee G\}_{n=1}^\infty$ is sensitive.
\end{proof}

\begin{example}
  If $G=mK_n$ then $\sigma(G)=1$ and for $\ell \leq \alpha(G)=m$ we have $\sigma(\compl{K}_{\ell} \vee mK_n)=1$.  If instead $\ell > m$, then for $n \geq 2$ we know $\sigma_k(mK_n)=\left\lceil \frac{k}{m} \right\rceil$ by~\cref{k-sens-nH}, and in this case~\cref{n-cone gensensitivity} could be used to compute $\sigma(\compl{K}_{\ell} \vee mK_n)$.
\end{example}

\begin{example}
  \label{Km-join-PnCn}
  We consider both $\compl{K}_m \vee G_n$ for $G_n=P_n$, referred to as \textit{agave graphs} in~\cite{Estrada-16}, and $G_n=C_n$, referred to as \textit{generalized wheel graphs} in~\cite{Buckley-Harary-88}.
  We take $m \geq 3$ and $n \geq 4$ or $n \geq 5$ respectively in order to avoid anomalies in the maximum degree and sensitivity for small graphs. With these assumptions on $m$ and $n$ we have
  \[\Delta(\compl{K}_m \vee P_n)
    =\Delta(\compl{K}_m \vee C_n)=\max\left\{m+2,n \right\}.\]
  Also,
  $\alpha(\compl{K}_m \vee G_n)=\max\left\{m,\alpha(G_n)\right\}$ where $\alpha(P_n)=\left\lceil \frac{n}{2} \right\rceil$ and $\alpha(C_n)=\left\lfloor \frac{n}{2} \right\rfloor$
  and
  $\sigma(P_n)=\sigma(C_n)=1$.

  When $m \leq \alpha(G_n)$, i.e., $m \leq \left\lceil \frac{n}{2} \right\rceil$ for $G_n=P_n$ or $m \leq \left\lfloor \frac{n}{2} \right\rfloor$ for $G_n=C_n$, which ensures that $\alpha(\compl{K}_m \vee G_n)=\alpha(G_n)$,
  then by~\cref{n-cone gensensitivity}~(\ref{item3}), it follows that $\sigma(\compl{K}_m \vee G_n)=\sigma(G_n)=1$.  Hence the families $\{\compl{K}_m \vee G_n\}_{n=1}^{\infty}$ for $G_n=P_n$ or $G_n=C_n$ are insensitive.

  If $\alpha(G_n)<m<n$ so that~\cref{n-cone gensensitivity}~(\ref{item2}) applies, we use that $\Delta(P_n)=\Delta(C_n)=2$ to note that both $\sigma_k(P_n)\leq 2$ and $\sigma_k(C_n) \leq 2$ for all $k \geq 1$ (in particular for $k=m+1-\alpha(G_n)$) and conclude that $\sigma(\compl{K}_m \vee P_n)\leq 2$ and $\sigma(\compl{K}_m \vee C_n) \leq 2$. Hence if $m$ and $n$ both increase while maintaining $\left\lceil \frac{n}{2} \right\rceil<m<n$ or $\left\lfloor \frac{n}{2} \right\rfloor<m<n$ respectively, then the resulting families of graphs $\compl{K}_m \vee P_n$ and $\compl{K}_m \vee C_n$ are insensitive.

  When $m \geq n$ then~\cref{n-cone gensensitivity}~(\ref{item1}) and~\cref{sens-n-cones} imply that the families $\{\compl{K}_m \vee P_n\}_{m=1}^{\infty}$ and $\{\compl{K}_m \vee C_n\}_{m=1}^{\infty}$ are sensitive and for $m$ sufficiently large, the sensitivity will be at least $m+1-n$.

In summary, in~\cref{Kn-join-G} and~\cref{Km-join-PnCn} we have shown that the family
$\{\compl{K}_m\vee G_n\}_{n=1}^\infty$ is insensitive for $G_n=K_n$, $G_n=P_n$, and $G_n=C_n$, while $\{\compl{K}_m\vee G_n\}_{m=1}^\infty$ is sensitive for $G_n=K_n$, $G_n=P_n$, and $G_n=C_n$.
\end{example}

We now apply~\cref{n-cone gensensitivity}~(\ref{item3}) and~\cref{complete-bipartite-sens} to obtain the sensitivity of an arbitrary complete multipartite graph.

\begin{proposition}\label{complete-multipartite-sens}
  If $n_1 \geq \cdots \geq n_k$,
  then
  \[
    \sigma(\compl{K}_{n_1} \vee \cdots \vee \compl{K}_{n_k}) = \sigma(\compl{K}_{n_1} \vee \compl{K}_{n_2})= \begin{cases}
      n_1+1-n_2                                & \text{ if $n_2 \leq \left\lfloor \frac{n_1+1}{2}\right\rfloor$\,,} \\
      \left\lceil \frac{n_1+1}{2} \right\rceil & \text{ if $n_2 > \left\lfloor \frac{n_1+1}{2}\right\rfloor$.}
    \end{cases}
  \]
\end{proposition}

\begin{proof}
  Consider the graph $\compl{K}_\ell \vee \compl{K}_m \vee \compl{K}_n$, where $\ell \leq m \leq n$ and view this graph as the join of $\compl{K}_\ell$ and $G=\compl{K}_m \vee \compl{K}_n$. Observe that $\alpha(\compl{K}_\ell \vee \compl{K}_m \vee \compl{K}_n) = n$ and consider how $\ell$ and $m$ compare to $\frac{n+1}{2}$.

  \textbf{Case 1:} $\ell \leq m \leq \frac{n+1}{2}$. By \cref{complete-bipartite-sens}, $\sigma(\compl{K}_m \vee \compl{K}_n) = n+1-m$. Since $\ell \leq m$,~\cref{n-cone gensensitivity}~(\ref{item3}) implies that $\sigma(\compl{K}_\ell \vee \compl{K}_m \vee \compl{K}_n) = \min \{n+1-m,n+1-\ell\} = n+1-m$.

  \textbf{Case 2:} $\ell \leq \frac{n+1}{2} < m$. By \cref{complete-bipartite-sens}, $\sigma(\compl{K}_m \vee \compl{K}_n) = \left\lceil \frac{n+1}{2} \right\rceil$. By~\cref{n-cone gensensitivity}~(\ref{item3}), $\sigma(\compl{K}_\ell \vee \compl{K}_m \vee \compl{K}_n) = \min \{\left\lceil \frac{n+1}{2} \right\rceil, n+1-\ell\} = \left\lceil \frac{n+1}{2} \right\rceil$.

  \textbf{Case 3:} $\frac{n+1}{2} < \ell \leq m$. By \cref{complete-bipartite-sens}, $\sigma(\compl{K}_m \vee \compl{K}_n) = \left\lceil \frac{n+1}{2} \right\rceil$, so by~\cref{n-cone gensensitivity}~(\ref{item3}), $\sigma(\compl{K}_\ell \vee \compl{K}_m \vee \compl{K}_n) = \left\lceil \frac{n+1}{2} \right\rceil$.

  We have shown that $\sigma(\compl{K}_\ell \vee \compl{K}_m \vee \compl{K}_n) = \sigma(\compl{K}_m \vee \compl{K}_n)$. It follows by induction, and then reordering for convenience, that if $n_1 \geq \cdots \geq n_k$, then
  \[
    \sigma(\compl{K}_{n_1} \vee \cdots \vee \compl{K}_{n_k}) = \sigma(\compl{K}_{n_1} \vee \compl{K}_{n_2}).
  \]
  Then~\cref{complete-bipartite-sens} can be applied to complete the result.
\end{proof}

A complete multipartite graph is regular if and only if each part has the same cardinality since if there were two parts with a different number of vertices, then vertices in those parts would have different degrees. We denote by $K_n^m$ the complete regular multipartite graph with $m$ parts of order $n$ and find its sensitivity by applying~\cref{complete-multipartite-sens} to find $\sigma(K_n^m)=\sigma(K_n^2)=\sigma(\compl{K}_n \vee \compl{K}_n)=\left\lceil \frac{n+1}{2} \right\rceil$.  Since also $\Delta(K_n^m)=\Delta(\compl{K}_n \vee \cdots \vee \compl{K}_n)=(m-1)n$, this provides another doubly indexed family with both sensitive and insensitive subfamilies.

\begin{corollary}
  \label{complete-reg-multipartite}
  The sensitivity of the complete regular multipartite graph with $m$ parts of order $n$ is
  $\sigma(K_n^m)=\sigma(\compl{K}_n \vee \cdots \vee \compl{K}_n) = \left\lceil \frac{n+1}{2} \right\rceil$.
  The family $\{K_n^m\}_{n=1}^\infty$ is sensitive and the family $\{K_n^m\}_{m=1}^\infty$ is insensitive with a constant sensitivity that can be made arbitrarily large by fixing a sufficiently large number of vertices in each part.
\end{corollary}

The complete multipartite graphs in~\cref{complete-multipartite-sens}, some generalizations of windmill graphs in~\cref{windmill}, and other graph families can naturally be described in a unified way via the generalized join construction, in which several graphs are joined together according to the structure of a graph $G$.

\begin{definition}
\label{genjoin}
  Given a graph $G$ with vertices $v_i$ and a corresponding graph $H_i$ for $i=1,\dots,n$, the \textit{generalized join of $H_1, H_2, \dots H_n$ over $G$}, denoted $G[H_1,\dots,H_n]$, is the graph obtained from the disjoint union of the graphs $H_1,\dots,H_n$ by adding edges joining each vertex in $H_i$ to each vertex in $H_j$ when $v_i$ is adjacent to $v_j$ in $G$.
\end{definition}

The generalized join is also referred to as the \textit{joined union}~\cite{Chudnovsky-24} or \textit{Zykov sum}~\cite{Liao-Aziz-Alaoui-Hou-22}.  Many common constructions can be realized as other special cases of generalized joins. For example, $K_2[G_1,G_2]=G_1 \vee G_2$ is the standard join, $G[H,\dots,H]=G \circ H$ is the \textit{lexicographic product} or \textit{graph composition} of $G$ and $H$, and the complete multipartite graphs in~\cref{complete-multipartite-sens} are generalized joins of empty graphs over a complete graph, i.e., of the form $K_m[\compl{K}_{n_1},\dots,\compl{K}_{n_m}]$.

The windmill graphs $W_{m,n}=K_1 \vee mK_n$ in~\cref{windmill}, which can be viewed as the generalized join $S_{m+1}[K_1,K_n,\dots,K_n]$, were used in~\cite{Estrada-16} to explore divergence between often-used transitivity indices and proposed as appropriate models for citation and collaboration networks that exhibit similar divergence properties. 
In order to extend and expand these results, in~\cite{Kooij-19} Kooij considered both what he called type I generalized windmill graphs $K_m\vee nH=S_{n+1}[K_m,H,\dots,H]$ and type II generalized windmill graphs $\compl{K}_m \vee nH=S_{n+1}[\compl{K}_m,H,\dots,H]$, as well as a type~III graph that is actually a corona - see~\cref{corona-def}. The $W_{m,n}$ and Types~I and~II in~\cite{Kooij-19} are all subsumed as particular instances of $G[H_0,H_1,\dots,H_n]$ with $G=S_{n+1}$.  The generalized join or joined union $G[H_1,\dots,H_n]$ for any graph $G$ of order $n$ is described as a fourth family of generalized windmill graphs in~\cite{Liao-Aziz-Alaoui-Hou-22}, where it is noted they more closely relate to real-life transportation networks, and in~\cite{Chudnovsky-24} it is pointed out that the graphs $H_1,\dots,H_n$ can be considered to describe the layers of a multilayer network with $G$ controlling how the layers are joined.  We consider a few cases of generalized joins for which we can determine sensitivity.

\begin{example}
  \label[example]{genwindmill}
  The generalized join of a set of graphs $H_1, H_2,\dots, H_m$ over an empty graph $\compl{K}_m$ is the disjoint union of the $H_i$.  The sensitivity is thus
  \[\sigma(\compl{K}_m[H_1,\dots,H_m])=\begin{cases}
  \min \{\sigma(H_i) \mid H_i \hbox{ is nonempty}\} & \hbox{if at least one $H_i$ is nonempty,} \\
  \sum\limits_{i=1}^m \sigma(H_i) & \hbox{if all $H_i$ are empty.}
  \end{cases}
  \]
because adding an isolated vertex to a nonempty graph does not change the sensitivity.  
  When all of the $H_i$ are nonempty this generalized join corresponds to the partition of a graph into connected components, which is referred to as the \textit{parallel decomposition} (see~\cref{Concluding remarks}). 
\end{example}

\begin{example}
  \label{series}
  For the generalized join over a complete graph, $G=K_m[H_1,\dots,H_m]$, if $G=K_m$ and all $H_i$ are equal to $H$, then $G[H_1,\dots,H_m]=K_m \circ H=H \vee \cdots \vee H$ is both a lexicographic product and a successive join.  As seen in~\cref{doublegraph}, which came from successively applying~\cref{sensitivity-join-edges},
  \[\sigma(K_m \circ H)=\sigma(H \vee H \vee \cdots \vee H)=\min\left\{\sigma(H),\left\lceil \frac{\alpha(H)+1}{2} \right\rceil\right\}.\]
  For arbitrary $H_1,\dots,H_n$~\cref{sensitivity-join-edges} still applies successively, but even if the $H_i$ were ordered to be monotonic in $\alpha_i$, computing $\sigma(K_m[H_1,\dots,H_m])$ could require $k$-sensitivities of some $H_i$ and switching between using~\cref{sensitivity-join-edges}~(\ref{join-1}) and~(\ref{join-2}) if there is not a consistent comparison between the larger of the independence numbers and the order of the graph with smaller independence number.
Being able to fully compute $\sigma(K_m[H_1,\dots,H_m])$ would address the \textit{series decomposition}, which describes the partition of $\compl{G}$ into connected components $\compl{H}_1,\dots,\compl{H}_m$ (see~\cref{Concluding remarks}).
\end{example}

\begin{example}
\label{TypeII}
For the generalized windmill graphs of type I in~\cite{Kooij-19}, by~\cref{Kn-join-G} and~\cref{genwindmill} we have $S_{n+1}[K_m,H_1,\dots,H_n]=K_m \vee (H_1+\dots+H_n)$ and hence
  \[\sigma(S_{n+1}[K_m,H_1,\dots,H_n])
    =\begin{cases}
  \min \{\sigma(H_i) \mid H_i \hbox{ is nonempty}\} & \hbox{if at least one $H_i$ is nonempty,} \\
  \sum\limits_{i=1}^m \sigma(H_i) & \hbox{if all $H_i$ are empty.}
  \end{cases}\]

  For certain Type II generalized windmill graphs $S_{n+1}[\compl{K}_m,H,\dots,H]=\compl{K}_m\vee nH$ we can use~\cref{n-cone gensensitivity} to more explicitly compute sensitivity when it is possible to compute $\sigma_k(nH)$ using~\cref{k-sens-nH}.
  Assume $H$ is not empty because if $H=\compl{K}_{\ell}$ for some $\ell \geq 1$ then $\sigma(\compl{K}_m \vee nH)=\sigma(\compl{K}_m \vee \compl{K}_{n\ell})$ is known by~\cref{complete-bipartite-sens}.
  Note that in general
  \begin{align*}
    \Delta(\compl{K}_m\vee nH) & =\max\{n|V_H|,m+\Delta(H)\} \hbox{ and } \\
    \alpha(\compl{K}_m\vee nH) & =\max\{m,n\alpha(H)\}.
  \end{align*}
  Here we consider only the case where $m \leq n\alpha(H)$ so that $\alpha(\compl{K}_m \vee nH)=n\alpha(H)$ and hence by~\cref{n-cone gensensitivity}~(\ref{item3})
  we have
  \[\sigma(\compl{K}_m\vee nH)=
    \begin{cases}
      \min \{\sigma(H), n\alpha(H)+1-m  \} \                                 & \text{if $m \leq \left\lfloor \frac{n\alpha(H)+1}{2} \right\rfloor$\,,} \\
      \min \{\sigma(H), \left\lceil \frac{n\alpha(H)+1}{2} \right\rceil\} \  & \text{if $m > \left\lfloor \frac{n\alpha(H)+1}{2} \right\rfloor$.}
    \end{cases}
  \]
  Thus the family $\{\compl{K}_m \vee nH\}_{n=1}^\infty$ is insensitive.  In contrast, the family $\{\compl{K}_m \vee nH\}_{m=1}^\infty$ is sensitive by~\cref{sens-n-cones}.

  In general for $S_{n+1}[\compl{K}_m,H_1,\dots,H_n]=\compl{K}_m \vee (H_1+ \cdots +H_n)$,
  when
  \[m \leq \alpha(H_1+ \cdots +H_n)=\sum_{i=1}^n \alpha(H_i)\]
 \cref{n-cone gensensitivity}~(\ref{item3}) provides $\sigma(S_{n+1}[\compl{K}_m,H_1,\dots,H_n])$ in terms of $\sigma(H_1+ \cdots +H_n)=\min\limits_{1 \leq i \leq n} \{\sigma(H_i)\}$.
  For example, if any of the $H_i$ have sensitivity $1$ then $\sigma(S_{n+1}[\compl{K}_m,H_1,\dots,H_n])=1$ and the family where $m \to \infty$ is insensitive.
\end{example}

\begin{corollary}
  \label{n-cone-sens-families}
  Let $m\geq 1$ and $\{G_n\}_{n=1}^\infty$ be a sensitive family. Then the family $\{\compl{K}_m \vee G_n\}_{n=1}^\infty$ is also sensitive.
\end{corollary}

\begin{proof}
  Since $\sigma(G_n) \to \infty$ as $n \to \infty$, we observe that $\alpha(G_n) \to \infty$ as $n \to \infty$. So there exists an $N_1 > 0$ such that for any $n \geq N_1$ we have $\alpha(G_n) \geq m$. This means that for each $n \geq N_1$, $\sigma(\compl{K}_m \vee G_n)$ equals either $\left\lceil \frac{\alpha(G_n)+1}{2} \right\rceil$, $\alpha(G_n) + 1 - m$, or $\sigma(G_n)$ by~\cref{n-cone gensensitivity}~(\ref{item3}).

  Now let $C > 0$. Each of the quantities $\left\lceil \frac{\alpha(G_n)+1}{2} \right\rceil$, $\alpha(G_n) + 1 - m$, and $\sigma(G_n)$ become arbitrarily large as $n$ increases. Thus, there exists an $N_2 \geq N_1$ such that for any $n \geq N_2$, each of these three quantities is greater than $C$ and hence $\sigma(\compl{K}_m \vee G_n) \geq C$, which completes the proof.
\end{proof}

Sensitive families can thus be constructed by either taking a sequence of $n$-cones of any graph, as in~\cref{sens-n-cones}, or by fixing $m$ and taking the $m$-cone of a family of sensitive graphs, as in~\cref{n-cone-sens-families}.  For instance, 
$\{\compl{K}_m \vee Q_{2n+1}^+\}_{n=1}^\infty$, $\{\compl{K}_m \vee K_{1,n}\}_{n=1}^\infty$, and $\{\compl{K}_m \vee K_{1,1,n}\}_{n=1}^\infty$
are all sensitive.
Insensitive families can be constructed by joining any graph with the sequence of complete graphs, as in~\cref{Kn-join-G}, or by building certain types of generalized joins - e.g., a sequence of Type II generalized windmill graphs with an increasing number of pendant graphs, or in which one of the pendant graphs has sensitivity $1$ as in~\cref{TypeII}. Other operations, like taking the $1$-cone of a graph or doubling any graph, preserve sensitivity or insensitivity.

\subsection{Joining an empty graph to a regular nonempty graph}
\label{reg-n-cones-section}

If $m$ is large relative to the number of vertices of a graph $G$, then $\sigma(\compl{K}_m \vee G)$ is given by~\cref{n-cone gensensitivity}. The following proposition shows that when $G$ is regular and $m$ is sufficiently large, there is a simple formula for $\sigma(\compl{K}_m \vee G)$.

\begin{theorem}
  \label{n-cone-lower-bound}
  Let $G$ be a regular graph with degree $d$ on $n$ vertices.  Then provided $m\geq 2n-d-1$ we have $\sigma(\compl{K}_m \vee G) = m-n+d+1$.
\end{theorem}

\begin{proof} 
  Since $m \geq n \geq \alpha(G)$, we have that $\alpha(\compl{K}_m \vee G)=\max\{m,\alpha(G)\}=m$.
  Consider an arbitrary subgraph $H$ induced by a subset of $m+1$ vertices consisting of all $n$ vertices in $G$ and $m-n+1$ vertices in $\compl{K}_m$. Each vertex in $H \cap \compl{K}_m$ has degree $n$ and each vertex in $H \cap G$ has degree $m-n+d+1$.  Thus for all such $H$ we have $\Delta(H)=m-n+d+1$ since $m-n+d+1 \geq n$.

  We claim that $\sigma(\compl{K}_m \vee G) = m - n + d + 1$. Since the above paragraph shows that any subgraph on $m+1$ vertices that contains all $n$ vertices of $G$ has maximum degree $m - n + d + 1$, it suffices to show that any other subgraph has larger maximum degree. Let $H$ be an induced subgraph on $m+1$ vertices which does not include all $n$ vertices of $G$.

  Let $k$ be the number of vertices of $G$ not included in $H$ and let $v \in H$. If $v \in \compl{K}_m$, then $\deg_H (v) = n - k$.  We will show if $v \in G$, then $\deg_H (v)\geq m - n + d + 1$.
  Since we assumed $m\geq 2n-d-1$, then $m - n + d + 1\geq n>n-k$ and hence it will follow that $\Delta(H)\geq m+1-n+d$.

  If $v \in G$, then
  \[
    \deg_H (v) = m + 1 - n + k + \deg_{H \cap G} (v).
  \]
  Since the number of neighbors of $v$ in $G$ that are not included in $H$ is $d - \deg_{H \cap G}(v)$ we have $k - (d - \deg_{H \cap G}(v)) \geq 0$ and can rewrite
  \[
    \deg_H(v)= m - n + d + 1 + k - (d-\deg_{H \cap G}(v))
  \]
  to observe that
  \[
    \deg_H(v)\geq m - n + d + 1.
  \]
  This completes the proof.
\end{proof}

\begin{example}
  In the case when $m\geq n$~\cref{n-cone-lower-bound} provides another proof that the complete split graphs $\{\compl{K}_m \vee K_n\}_{m=1}^\infty$ form a sensitive family since $\sigma(\compl{K}_m \vee K_n)=m$ (as shown in~\cref{Kn-join-G} for any $m,n \geq 1$).
\end{example}

\begin{example}
  For generalized wheel graphs $\compl{K}_m \vee C_n$, with $m \geq 2n-3$ it follows immediately from~\cref{n-cone-lower-bound} that $\sigma(\compl{K}_m \vee C_n)=m-n+3$, providing more precise information than obtained from~\cref{sens-n-cones} in the portion of~\cref{Km-join-PnCn} about sensitive families.
\end{example}

\begin{example}
  When $G$ is a regular graph of degree $d$ with $n$ vertices and $m \geq 2n-d-1$, \cref{n-cone-lower-bound} yields $\sigma(\compl{K}_m \vee G)=m-n+d+1\geq n$ and hence $\{\compl{K}_m \vee G\}_{m=1}^{\infty}$ is sensitive, which is consistent with~\cref{sens-n-cones} but provides the exact sensitivity in exchange for the conditions on $G$ and $m$.
In particular, if $G=nH$ where $H$ is a regular graph of degree $d$ and order $\ell$, then $\sigma(\compl{K}_m \vee nH)=m-n\ell+d+1$ for the family of generalized windmill graphs of type II and $\{\compl{K}_m \vee nH\}_{m=1}^{\infty}$ is sensitive.
\end{example}

\section{Sensitivity and stable blocks}
\label{stable blocks}
In this section we consider when a graph $G$ has a decomposition into sets called \textit{stable blocks}.
These are defined by Larson in~\cite{Larson-12} and allow isolating maximum independent sets and describing sensitivity.

\begin{definition}
  \label{sb-def}
  Let $G = (V,E)$ be a graph and let $S$ be a nonempty proper subset of $V$. The \emph{border} of $S$, denoted $\mathrm{bord}(S)$, is the set of vertices in $S$ adjacent in $G$ to at least one vertex in $V \setminus S$; the \emph{interior} of $S$, denoted $\mathrm{int}(S)$, is $S \setminus \mathrm{bord}(S)$; and the set $S$ is called a \emph{stable block} if $\alpha(G[S]) = \alpha(G[\mathrm{int}(S)])$.
\end{definition}
Note that if $S$ is a stable block, then $\mathrm{int}(S)$ is not the empty set.
Stable blocks were used to efficiently reduce independence number calculations to subgraphs in~\cite{Larson-11}. We use them to bound sensitivity in~\cref{single stable block}, and in~\cref{stable-block-partition} to compute sensitivity precisely in terms of the blocks when there is a full partition of $V_G$ into such sets.  We also define vertex identification and use it as a means to construct a stable block in~\cref{gluing-root-sens}, repeat this to bound the sensitivity of a rooted product in~\cref{RootedProducts}, and indicate how to construct a graph with a partition into stable blocks in~\cref{StableBlockDecomp}.  As an application we compute the sensitivity of a corona of two graphs and of families of successive corona graphs.

\subsection{Single stable block}
\label{single stable block}

We first consider graphs in which there exists a single stable block.
Note that this need not occur; for example a complete graph has no stable blocks.
But since a stable block in a graph $G$ is a proper subgraph of $G$, when one does exist it may be easier to compute its sensitivity than $\sigma(G)$, which is then bounded above by the sensitivity of the interior of the stable block.

\begin{proposition}\label{stable-block-sens}
  Let $G$ be a graph, let $H$ be an induced subgraph such that $V_H$ is a stable block. Assume that $G[\mathrm{int}(H)]$ is nonempty. 
  Then $\sigma(G) \leq \sigma(G[\mathrm{int}(H)])$.
\end{proposition}

\begin{proof}
  Let $G$ and $H$ be as in the statement and $A \subset V_G$ be a maximum independent set of vertices in $G$. We first note that $H$ contains $\alpha(G[\mathrm{int}(H)]) = \alpha(H)$ vertices of $A$. For if not, there is a larger independent subset of $G$ which agrees with $A$ outside of $V_H$, and forms a maximum independent subset of $G[\mathrm{int}(H)]$ inside of $H$.

  Next, let $A'$ be a maximum independent subset of $V_G$ with no vertices in the border of $H$; such a subset could be constructed by starting with a maximum independent subset of $G[\mathrm{int}(H)]$, then adding the vertices from $A - V_H$.

  By utilizing an $(\alpha(H)+1)$-induced subgraph in $G[\mathrm{int}(H)]$ with maximum degree $\sigma(G[\mathrm{int}(H)])$, construct $W$ as an $(\alpha(G) + 1)$-induced subgraph of $G$ such that
  \begin{enumerate}
    \item $V_W \cap (V_G - \mathrm{int}(H)) = A' \cap (V_G - \mathrm{int}(H))$, and
    \item $\Delta(W) = \sigma(G[\mathrm{int}(H)])$.
  \end{enumerate}
  These conditions ensure that the only edges of $W$ lie inside $G[\mathrm{int}(H)]$. This $W$ is an induced subgraph on $\alpha(G) + 1$ vertices, so $\sigma(G) \leq \Delta(W) = \sigma(G[\mathrm{int}(H)])$.
\end{proof}

\begin{remark}
The case of a star graph $G=K_1 \vee \overline{K}_n$ with $H$ a proper subgraph containing the center vertex shows that the conclusion of~\cref{stable-block-sens} need not hold when $G[\mathrm{int}(H)]$ is empty.
\end{remark}

The process of vertex identification in a disjoint union of two graphs gives one possible way to produce a stable block and can be generalized to rooted products and coronas - see~\cref{RootedProducts}.

\begin{definition}
  \label{identify point}
  Given two graphs $G_1$ and $G_2$ with distinguished vertices $v_1$ and $v_2$ respectively, the \textit{identification of $G_1$ and $G_2$ at $v_1$ and $v_2$}, denoted by $(G_1+G_2)/\{v_1,v_2\}$, is the disjoint union $G_1+G_2$ with the vertices $v_1$ from $G_1$ and $v_2$ from $G_2$ replaced by a single vertex that is adjacent to all neighbors of $v_1$ in $G_1$ and of $v_2$ in $G_2$.
\end{definition}

\begin{example}
  \label{dandelion}
  The family of \textit{dandelion graphs} is $D_{m,n}=(K_1 \vee \compl{K}_m + P_n)/\{v_c,v_e\}$, where the center vertex $v_c$ of a star graph $K_1 \vee \compl{K}_m$ is identified with an end vertex $v_e$ of path $P_n$. To avoid the star graphs $D_{m,1}=K_1 \vee \compl{K}_{m}$ and $D_{m,2}=K_1 \vee \compl{K}_{m+1}$ and the anomalous case of $\sigma(D_{m,4})=2$ for $m \geq 2$ assume $m \geq 1$ and $n>4$. Since the identified vertex cannot be included in a maximum independent set, $\alpha(D_{m,n})=m+\left\lfloor \frac{n}{2} \right\rfloor$ rather than $\alpha(K_1 \vee \compl{K}_m) + \alpha(P_n)$.
  Starting at the unidentified end of $P_n$ and alternately including two vertices and omitting a vertex yields that $\sigma(D_{m,n})=\sigma(P_n)=1$. Since $\Delta(D_{m,n})=m+1$, then $\{D_{m,n}\}_{m=1}^{\infty}$ is insensitive while $\{D_{m,n}\}_{n=1}^{\infty}$ is not.
\end{example}

\begin{corollary}
  \label{gluing-root-sens}
  Let $H$ and $H'$ be graphs rooted at $v$ and $v'$ respectively. Let $G=(H+H')/\{v,v'\}$. Assume that there exists a maximum independent set of vertices in $H$ which does not contain $v$ and that $G[\mathrm{int}(H)]$ is nonempty. Then $\sigma(G) \leq \sigma(G[\mathrm{int}(H)])$.
\end{corollary}

\begin{proof}
  Since $G[\mathrm{int}(H)]$ is nonempty, by \cref{stable-block-sens} it suffices to show that the set  of vertices of $G[H]$ is a stable block in $G$.
  As a result of the vertex identification that produces $G$, the subgraph $G[H]$ has the vertex $v$ as its border.  Since $v$ lies outside of some maximum independent subset of $H$, we know $\alpha(H) = \alpha(G[\mathrm{int}(H)])$ and the set of vertices of $G[H]$ is a stable block in $G$.
\end{proof}

\begin{example}
  We can produce examples by applying~\cref{gluing-root-sens} to any graph $H'$ and any $1$-cone $H$ whose interior has at least one edge, identified at the cone point of $H$ and any vertex of~$H'$.
  For instance, if $H$ is any of the $1$-cones $K_1 \vee G_n$ with $G_n=K_n$ for $n \geq 2$, $G_n=P_n$ for $n \geq 4$, or $G_n=C_n$ for $n \geq 5$, as in~\cref{Kn-Pn-Cn}, then the resulting graph family will have sensitivity $1$ by~\cref{1-cone-examples-insens} and~\cref{gluing-root-sens}.

  Using $H=K_{1,1,n}=K_1 \vee K_{1,n}$ or $K_{1,1,1,n}=K_1 \vee (K_1 \vee K_{1,n})$ with $n \geq 2$ and~\cref{star} or~\cref{1-cone-example-sens}, or using $H=K_1 \vee (\compl{K}_n \vee K_m)$ with $m,n \geq 1$ and~\cref{Kn-join-G}, produces graph families whose sensitivity is bounded above by $n$.
  In the last case since $\Delta(K_1 \vee (\compl{K}_m\vee K_n))=m+n$ a resulting family will be insensitive as $m \to \infty$.

  Similarly, applying~\cref{gluing-root-sens} at the $K_1$ vertex in $H=K_1 \vee K_n^m$ for $m,n \geq 1$ and for any choice of $H'$ and identified vertex yields a graph family with sensitivity at most $\left\lceil \frac{n+1}{2} \right\rceil$ by~\cref{complete-reg-multipartite}.  Since $\Delta(K^m_n)=(m-1)n$, letting $m \to \infty$ results in an insensitive family.
\end{example}

\begin{example}
  \label{cycle-star}
  For the star graph $S_n=K_1 \vee \compl{K}_n$ with $\{v\}$ the distinguished vertex in $K_1$, and the cycle graph $C_m$ with $m \geq 3$ and any distinguished vertex $\{u\}$, the \textit{cycle-star graph}~\cite{Clark-85} is $CS_{m,n}=(C_m+S_n)/\{u,v\}$.  It has $\alpha(CS_{m,n})=\left\lfloor \frac{m}{2} \right\rfloor+n$ and $\Delta(CS_{m,n})=n+2$.  Since there is a maximum independent subset in $C_m$ not containing $u$ and $CS_{m,n}[\mathrm{int}(C_m)]=C_m\setminus\{u\}$ is a nonempty path, we conclude that $\sigma(CS_{m,n})=\sigma(CS_{m,n}[\mathrm{int}(C_m)])=1$ for $m\geq 5$ (so that the path has length at least $4$) and the family $\{CS_{m,n}\}_{n=1}^{\infty}$ is insensitive.
\end{example}

\begin{example}
  The family of \textit{pineapple graphs} is $P_{m,n}=((K_1 \vee \compl{K}_m)+K_n)/\{u,v\}$ where $u$ is the distinguished $K_1$ vertex in $K_1 \vee \compl{K}_m$  and $v$ is any vertex of $K_n$.  We take $m \geq 1$ and $n \geq 3$ so that the graph is not a star. Using $H=K_n$ in~\cref{gluing-root-sens} yields that pineapple graphs also have sensitivity $1$.  Since the maximum degree is $m+n-1$ both $\{P_{m,n}\}_{m=1}^{\infty}$ and $\{P_{m,n}\}_{n=1}^{\infty}$ are insensitive.
\end{example}

The following definition used in~\cite{Larson-11} was originally given in~\cite{Zhang-90}.

\begin{definition}
  Let $I_c$ be an independent set of vertices in a graph $G$, and let $N(I_c)$ denote the neighborhood of $I_c$, consisting of all vertices adjacent to some vertex in $I_c$. If for each independent subset $J$ of $G$ we have $|I_c| - |N(I_c)| \geq |J| - |N(J)|$ then $I_c$ is called a \emph{critical independent set} and such a set of maximum cardinality is called a \emph{maximum critical independent set}.
\end{definition}

It is possible that the empty set is the only maximum critical independent set, as is the case when $G=K_n$ for all $n \geq 3$.  But if not, the Independence Decomposition Theorem of~\cite{Larson-11} produces a stable block.

\begin{proposition}
  \label{max-crit-indep-set}
  Let $I_c \subseteq G$ be a maximum critical independent set, and assume $I_c \neq \varnothing$. Then $X = I_c \cup N(I_c)$ is a stable block of $G$ or $X=V_G$.
\end{proposition}
\begin{proof}
  At the start of the proof of Theorem $2.4$ in \cite{Larson-11}, Larson proves that $I_c$ is a maximum independent subset of $G[X]$, so $\alpha(G[X]) = |I_c|$.

  Since $\mathrm{int}(X) \subseteq X$, we know that $\alpha(G[\mathrm{int}(X)]) \leq \alpha(G[X])$. Since $I_c \subseteq \mathrm{int}(X)$, we also know $\alpha(G[X]) \leq \alpha(G[\mathrm{int}(X)])$ and we conclude that if $X \subsetneq V_G$ then $X$ is a stable block in $G$.
\end{proof}

\begin{example}
\label{pineapple stable block}
Not every stable block arises as in~\cref{max-crit-indep-set}.  For instance, in the pineapple graphs $P_{m,n}=((K_1 \vee \compl{K}_m)+K_n)/\{u,v\}$, one can check that the vertices in $\compl{K}_m$ form the only critical independent set, so the vertices of $K_n$ form a stable block in $P_{m,n}$ that contains no maximum critical independent set, while the vertices in $K_1 \vee \compl{K}_m$ form a stable block that does arise as in~\cref{max-crit-indep-set}.  A similar situtation occurs if there is instead an edge between the central vertex in $K_1 \vee \compl{K}_m$ and one vertex in $K_n$ (the rooted product $K_2(K_1 \vee \compl{K}_m,K_n)$ - see~\cref{rp-def}).
\end{example}

Larson also proves that the set $X$ is independent of the choice of maximum critical independent set~\cite{Larson-11} and that a maximum critical independent set can be found in polynomial time~\cite{Larson-07}.  Hence when there exists a nonempty maximum critical independent set~\cref{max-crit-indep-set} provides an effective way to produce a stable block in a graph and use~\cref{stable-block-sens} to obtain an upper bound on sensitivity, which is most useful for proving a family is insensitive.

\subsection{Rooted products and coronas}
\label{RootedProducts}

Given a graph $G$ of order $n$ and a set of graphs $\mathcal{H}=\{H_1,\dots,H_n\}$ rooted at $v_1,\dots,v_n$ respectively, each root can be identified with a vertex of $G$ to obtain the \textit{rooted product $G(H_1,\dots,H_n)$ of $\mathcal{H}$ by $G$}. This construction was originally introduced in~\cite{Godsil-McKay-78} where the characteristic polynomial of its adjacency matrix was determined in terms of the characteristic polynomials for $H_1,\dots,H_n$.

\begin{definition}
  \label{rp-def}
  Consider a graph $G$ with vertices $u_1,\dots,u_n$ and a sequence $\mathcal{H}=H_1,\dots,H_n$ of graphs  rooted at the vertices $v_1,\dots,v_n$ respectively.  The \textit{rooted product} of $G$ with $\mathcal{H}$ is denoted $G(\mathcal{H})$ or $G(H_1,\dots,H_n)$ and is obtained by identifying pairs of vertices $u_i$ and $v_i$ for $1 \leq i \leq n$ in the disjoint union $G+H_1+\cdots +H_n$.
  When $H_1=H_2=\cdots=H_n=H$ we write $G(\mathcal{H})$ as $G(H)$.
\end{definition}
  
The construction in~\cref{identify point} and~\cref{gluing-root-sens} identifying vertices in two disjoint graphs $G$ and $H$ is equivalent to the rooted product $G(H,K_1,\dots,K_1)$.  We can repeatedly apply~\cref{gluing-root-sens} using a sequence of graphs $H_1, H_2, \dots, H_n$, resulting in an upper bound on the sensitivity of general rooted products and a useful tool for producing insensitive families or proving insensitivity.

\begin{corollary}
  \label{rooted-product}
  Let $G$ be a graph with $n$ vertices, let $1 \leq k \leq n$, and let $H_1,\dots,H_k$ be graphs, with $H_i$ rooted at a vertex $v_i$ that is not contained in some maximum independent subset of $H_i$ for $1 \leq i \leq k$. Assume that, for each $i$, the subgraph of $G(H_1,\dots,H_k,K_1,\dots,K_1)$ induced by $\mathrm{int}(H_i)$ is nonempty. Let $\sigma_i=\sigma(G[\mathrm{int}(H_i)])$ be the sensitivity of the subgraph induced by $\mathrm{int}(H_i)$. Then
  \[\sigma(G(H_1,\dots,H_k,K_1,\dots,K_1)) \leq \min_{1 \leq i \leq k}\{\sigma_i\}.\]
\end{corollary}

\begin{remark}
Since in applying~\cref{gluing-root-sens} to prove~\cref{rooted-product} the location of the distinguished vertex in the prior result does not matter,  the same bound will also hold if the vertex being identified  with $v_i$ in $H_i$ were in one of $H_1,\dots,H_{i-1}$, and for arbitrarily large $k$.  This yields the same bound in a much wider range of constructions, for which we do not attempt to devise notation.  For example, the sensitivity will be $1$ for any tree built by starting with a path and successively identifying an endpoint of a new path to a prior vertex, provided at least one of the paths used has length at least $4$. 
\end{remark}

\begin{example}
  For any graph $G$ with one or more distinguished vertices, using any $1$-cone $H$ whose interior is nonempty results in a wide variety of examples by applying~\cref{rooted-product} at the chosen vertices of $G$.  For instance, if $H$ is the join of $K_1$ and $K_n$, $P_n$, or $C_n$ for $n \geq 2$, $n \geq 4$, or $n \geq 5$ respectively, as in~\cref{Kn-Pn-Cn}, then the resulting graph will have sensitivity $1$, as will any tree built by successive vertex identification involving at least one path of length at least $4$. In particular the \textit{royal petunia graphs} of ~\cite{Latyshev-Kokhas-23} contain a copy of $K_1 \vee P_n$ as a stable block, and thus have sensitivity $1$.
\end{example}

If we identify a copy of $K_1 \vee H$ rooted at the vertex in $K_1$ to each vertex in $G$ then the resulting rooted product is called the \textit{corona} $G \odot H$ of $G$ and $H$. 
The corona operation was introduced in~\cite{Frucht-Harary-70} to produce a graph whose automorphism group could be easily described as a wreath product involving the automorphism groups of the component graphs.

\begin{definition}\label{corona-def}
  The \textit{corona} $G \odot H$ of a graph $G$ using another graph $H$ is the graph obtained by associating a copy $H_v$ of $H$ to each vertex $v$ of $G$, and adding edges from $v$ to each vertex of $H_v$.
\end{definition}

\begin{example}
\label{corona-inequality}
If $H$ is nonempty then $G \odot H=G(K_1 \vee H)$ will satisfy the hypotheses of~\cref{rooted-product} and each copy of $K_1 \vee H$ will be a stable block in $G \odot H$, so $\sigma(G \odot H) \leq \sigma(K_1 \vee H) = \sigma(H)$.
For example, the Type III windmill graphs in~\cite{Kooij-19} are $K_m \odot K_n$ and hence by~\cref{rooted-product} have sensitivity $\sigma(K_m \odot K_n)=1$. Since for $m \geq 2$ we have 
\[\Delta(K_m \odot K_n)=\max\{\Delta(K_m)+|V_{K_n}|,\Delta(K_n)+1\}=m-1+n,\]
all subfamilies are insensitive.
\end{example}

We will give $\sigma(G \odot H)$ precisely, even when $H$ is empty, in~\cref{corona-sens}.

\subsection{Stable block decomposition}
\label{StableBlockDecomp}

In contrast to modules and a maximal modular partition, stable blocks do not have to exist and can intersect, so a partition of a given graph into stable blocks might not exist.  However in the case where there is a full decomposition of a graph into stable blocks we obtain precise information about its sensitivity in terms of that of the blocks.  We also describe a construction that produces families to which~\cref{stable-block-partition} does apply.  The main idea is that if the pieces do not interact too much then they determine the sensitivity.

\begin{theorem}\label{stable-block-partition}
  If the vertex set of a graph $G$ can be partitioned into stable blocks $V_1,\dots,V_n$ such that $G[V_i]$ is nonempty for $1 \leq i \leq n$, then 
  \[\sigma(G)=\min\limits_{1\leq i \leq n} \sigma (G[V_i]).\]
\end{theorem}

\begin{proof}
  Let $V_1,\dots,V_n$ be a partition of the vertex set of $G$ into stable blocks such that for $1 \leq i \leq n$ we have $G_i := G[V_i]$ nonempty, and let $A_i$ be a maximum size independent set in $\mathrm{int}(G_i)$.  Suppose without loss of generality that $\sigma(G_1)=\min\limits_{1\leq i \leq n} \sigma (G_i)$.
  Clearly $\bigcup\limits_{i=1}^n A_i$ is an independent subset of $G$ of order $\sum\limits_{i=1}^n \alpha(G_i)$,
  but by the pigeonhole principle any larger subset intersects some $G_i$ in at least $\alpha(G_i)+1$ vertices, and thus is not independent in $G$. Hence $\alpha(G) = \sum\limits_{i=1}^n \alpha(G_i)$.

  Consider an induced subgraph $H$ of $G$ of order $1+\sum\limits_{i=1}^n \alpha(G_i)$.
  Since there is some $G_i$ that contains at least $\alpha(G_i)+1$ vertices of $H$, we know $\Delta(H) \geq \sigma(G_i) \geq \sigma(G_1)$ and thus $\sigma(G) \geq \sigma(G_1)$.

  Next, let $H_1$ be an induced subgraph of $G_1$ with $\Delta(H_1)=\sigma(G_1)$.  Consider the induced subgraph of $G$ consisting of $H_1$ together with $A_2,\dots,A_k$.  The only edges in that induced subgraph are between vertices in $H_1$ so the maximum degree is $\sigma(G_1)$.  Thus $\sigma(G)\leq \sigma(G_1)$, and we have shown $\sigma(G)=\min\limits_{1\leq i \leq n} \sigma (G_i)$.
\end{proof}

There is a natural way to construct graphs that decompose into stable blocks as in \cref{stable-block-partition}. Let $G$ be a graph, and for each vertex $v_i$ with $1\leq i \leq k$ of $G$ let $H_i$ be a nonempty graph corresponding to $v_i$ with vertex set $V_i$ and edge set $E_i$. Fix a maximum independent subset $A_i$ of each $H_i$. Now build a graph from $G$ and the $H_i$ so that the following properties hold:
\begin{enumerate}
  \item its vertex set is $\bigcup\limits_{i=1}^k V_i$\,,
  \item its edge set contains $\bigcup\limits_{i=1}^k E_i$\,, and
  \item if $v_i$ and $v_j$ are adjacent in $G$, then there exist vertices $u_i \in H_i \setminus A_i$ and $u_j \in H_j \setminus A_j$ which are adjacent in the resulting graph
\end{enumerate}
Then $V_1,\dots,V_k$ is a partition of the vertex set of the resulting graph into stable blocks.
Using this construction and examples in~\cref{joins} for which sensitivity is known one can build  graphs and graph families with known sensitivity.  By comparison it is not clear how the sensitivity of a fully generalized join depends on the sensitivities of its component graphs and this would be unreasonable to expect since any graph can be expressed as a generalized join in many ways.

In~\cref{corona-inequality}, we applied \cref{rooted-product} to see that $\sigma(G \odot H)$ is bounded above by $\sigma(H)$ whenever $H$ is nonempty.  But since the corona does have a partition into stable blocks even when $H$ is empty,~\cref{stable-block-partition} yields the following stronger result.

\begin{corollary}
\label{corona-sens}
  For graphs $G$ and $H$, 
  \[\sigma(G \odot H) = \sigma(K_1 \vee H)= \sigma(H).\]
\end{corollary}

\begin{proof}
  Observe that even when $H$ is empty, each copy of $K_1 \vee H$ is a stable block that induces a nonempty graph in the corona $G \odot H=G(K_1 \vee H)$. Since the copies of $K_1 \vee H$ form a partition of the vertices of $G \odot H$, by \cref{stable-block-partition} we have $\sigma(G \odot H) = \sigma(K_1 \vee H)$, which equals $\sigma(H)$.
\end{proof}

Using the examples in~\cref{defs} and~\cref{joins}, one can construct many different families of coronas and~\cref{corona-sens} reduces the computation of the sensitivity of even a complicated corona $G \odot H$ to the computation of $\sigma(H)$. For example, 
$\sigma(K_{1,1,\ell}\odot K^m_n)=\left\lceil\frac{n+1}{2}\right\rceil$ and $\sigma(K^m_n \odot K_{1,1,\ell})=\ell$, and 
if $G$ is any graph then $\sigma(G \odot K_n) = 1$ and $\sigma(G \odot \compl{K}_n) = n$.  These can also be used successively to build interesting families of coronas.  For instance, given a sequence of positive integers $n_1,n_2,\dots$ one could successively take the corona with $\compl{K}_{n_i}$ to produce a family exhibiting that sequence of sensitivities.  As a final particular example, for the successive corona family of complete regular bipartite graphs it is immediate by~\cref{corona-sens} that the sensitivity is
\[\sigma(((K_1^2 \odot K_2^2) \odot K_3^2) \cdots )\odot K_n^2) = \left\lceil \frac{n+1}{2} \right\rceil.\]

\section{Concluding remarks}\label{Concluding remarks}

We have described sensitivity under join and a partition into stable blocks and used our results to compute the sensitivity of several specific graph families and produce general constructions of sensitive and insensitive families.  
We note the following decomposition perspective on the results in~\cref{joins} in order to frame further directions related to expressing the sensitivity of other types of generalized joins in terms of the components. 

In a generalized join $G[H_1,\dots,H_n]$, each subgraph induced by $H_i$, $1 \leq i \leq n$ is a \textit{module}, i.e., a set of vertices that all have the same neighbors outside $H_i$.  Thus $H_1,\dots,H_n$ forms a modular partition of the full graph, and conversely any modular partition of a graph can be used to express it as a generalized join. In general it is possible to partition the vertices of a graph into modules in multiple ways.
The \textit{modular decomposition tree}
of a graph $G$ exhibits the nesting of its \textit{strong} modules (see~\cite{HMSZ-24}), uniquely determines $G$ (see~\cite{Brandstadt-Le-Spinrad-99}), and the maximal proper modules in that tree form a partition of the vertex set of $G$ that is referred to as a maximal modular partition of $G$ (see~\cite{Habib-Paul-10}).  
Reference~\cite{Brandstadt-Le-Spinrad-99} contains definitions, a list of several other terms that have been used for similar ideas in a variety of contexts, and precise statements of a modular decomposition theorem and related results.
For a brief primer and an interesting new generalization see~\cite{HMSZ-24} and for further details and a reference list see~\cite{Habib-Paul-10}.

When $G$ has more than one connected component $H_1,\dots,H_n$ the unique maximal modular partition is the \textit{parallel decomposition} $G=\compl{K}_n[H_1,\dots,H_n]$, for which we described the sensitivity in~\cref{genwindmill}.  When the complement $\compl{G}$ has more than one connected component $H_1,\dots,H_n$ then
$G=K_n[H_1,\dots,H_n]$ is the \textit{series decomposition}, for which we described sensitivity in the lexicographic product case, $K_n \circ H$, in~\cref{series}.
When $G$ and $\compl{G}$ are connected, then for a maximal modular partition $H_1,\dots,H_n$ of $G$ we can replace each $H_i$ with a single vertex and obtain a \textit{quotient} $Q$ that has no nontrivial modules and can be used to write $G$ as a generalized join $G=Q[H_1,\dots,H_n]$.  
For example, in the parallel and series decompositions the graphs $\compl{K}_n$ and $K_n$ are quotients of $G$ by $H_1,\dots,H_n$. This suggests limitations on finding the sensitivity of a generalized join explicitly in terms of its components, since in full generality this would amount to finding the sensitivity of a general graph.

To extend the range of examples and constructions it would help to expand the number of families for which the $k$-sensitivity is known or to develop other methods to understand $k$-sensitivity. The generalized join construction in~\cref{joins} and related maximal modular partition offer different possibilities and limitations compared to the stable block decomposition and construction in~\cref{stable blocks}, leading to several possible questions and further directions for investigation.

\begin{enumerate}
\item We list some questions related to generalized joins:
\begin{enumerate}
\item Is it practical to describe sensitivity in the series case, $\sigma(K_n[H_1,\dots,H_n])$, in terms of sensitivities of $H_1,\dots,H_n$ without the assumption that $H_1=H_2=\cdots=H_n$ as in~\cref{series}? 
\item If it is possible to completely describe $\sigma(K_n[H_1,\dots,H_n])$, can that be used together with the parallel case $\sigma(\compl{K}_n[H_1,\dots,H_n])$ in~\cref{genwindmill} to understand sensitivity of \textit{complement reducible graphs} or \textit{cographs}, which are those whose modular decomposition tree only contains series and parallel nodes, called a \textit{cotree}?
\item Are there other specific cases of generalized joins for which sensitivity can be computed?
\end{enumerate}

\item We list some questions related to stable blocks:
\begin{enumerate}
\item Are there other classes of stable blocks that do not arise as the union of a critical independent set and its neighbors beyond cases like $K_n$ in~\cref{pineapple stable block}? If so, can such stable blocks be described or characterized more generally (like in~\cref{max-crit-indep-set})?
\item Given a rooted product $G(H_1,\dots,H_n)$, how can we distinguish different types of stable blocks within it?
\item Can~\cref{stable-block-partition} be generalized to find sensitivity in terms of stable blocks when the stable blocks do not form a partition but are close to doing so?
\end{enumerate}
\end{enumerate}

More broadly it would be of interest to explore potential applications of sensitivity and connections with other graph invariants, find approximations to sensitivity that are of interest and are easier to compute, and investigate sensitivity under other types of products and operations.

\section{Acknowledgements}
We thank Davis Bolt and Gabriel Ory for computing support carried out in~SageMath~\cite{sagemath}.  The Idaho State University Career Path Internship Program provided funding for that computing support and some work by Jacob Tolman.

\section{Data Availability}
Code to compute graph sensitivity produced by Davis Bolt, code to generate graph families and sensitivity data produced by Gabriel Ory, and output for several graph families, are available at https://github.com/gabrielory/Graph-Sensitivity.

\bibliographystyle{plain}
\bibliography{Sensitivity}

\typeout{get arXiv to do 4 passes: Label(s) may have changed. Rerun}

\end{document}